\documentclass[12pt]{report}
\usepackage[letterpaper, left=1.5in, right=1in, top=1in, bottom=1in]{geometry}
\usepackage{amsmath,amssymb,amsthm,amsfonts}
\usepackage{algorithmic}
\usepackage{graphicx}
\usepackage{times} 
\usepackage{setspace}
\usepackage{textcomp}
\usepackage{xcolor}
\usepackage[bookmarks=true, hidelinks]{hyperref}
\usepackage[backend=bibtex]{biblatex}
\usepackage[utf8]{inputenc}
\usepackage[explicit]{titlesec}
\usepackage[titles]{tocloft}
\usepackage{appendix}
\usepackage{listings}
\usepackage{xy}
\xyoption{all}
\usepackage{epstopdf}
\usepackage{auto-pst-pdf}
\usepackage{tikz}
\usepackage{mathtools}
\usetikzlibrary{matrix}

\newtheorem{theorem}{Theorem}[section]
\newtheorem{lemma}[theorem]{Lemma}
\newtheorem{prop}[theorem]{Proposition}
\theoremstyle{definition}
\newtheorem{definition}{Definition}[section]
\newtheorem{example}{Example}
\newtheorem{corollary}{Corollary}
\theoremstyle{remark}

\numberwithin{equation}{section}
\newcommand{\Hall} {{\mathbf H_{\mathcal{A}}}}

\begin{document}
\doublespacing



\newcommand{\thesisTitle}{primitive elements in the hall algebra of a cyclic quiver}
\newcommand{\yourName}{Renda Ma}
\newcommand{\yourDept}{Mathematics}
\newcommand{\yourDegree}{Doctor of Philosophy}
\newcommand{\yourMonth}{December}
\newcommand{\yourYear}{2022}


\begin{titlepage}
\begin{center}

\begin{doublespacing}

\vspace*{0.5in}
\vfill
\textbf{\MakeUppercase{\thesisTitle}}\\
\vspace{3\baselineskip}
by\\
\vspace{3\baselineskip}
\yourName\\
\yourMonth{} \yourYear{}\\
\vspace{2\baselineskip}
A dissertation submitted to the \\ Faculty of the Graduate School of \\ the University at Buffalo, The State University of New York \\ in partial fulfilment of the requirements for the \\ degree of \\
\vspace{2\baselineskip}
\yourDegree\\
\vspace{3\baselineskip}
Department of \yourDept
\end{doublespacing}

\end{center}
\end{titlepage}
\currentpdfbookmark{Title Page}{TitlePage}

\pagenumbering{roman}
\setcounter{page}{2}
\begin{center}
\null
\vfill
\begin{doublespace}
Copyright by \\ 
Renda Ma \\ 
2022 \\
\end{doublespace}
\end{center}

\addcontentsline{toc}{chapter}{Acknowledgments}
\chapter*{Acknowledgments}


I would like to thank Professor Yiqiang Li for his patience, guidance and support towards me for my studies at University at Buffalo.  I would also like to thank Professor Joseph Hundley and Professor David Hemmer for their comments and help on my thesis.

\renewcommand{\cftchapdotsep}{\cftdotsep}
\renewcommand\contentsname{Table of Contents}
\begin{singlespace}
\tableofcontents
\end{singlespace}
\currentpdfbookmark{Table of Contents}{TOC}
\clearpage

\addcontentsline{toc}{chapter}{Abstract}
\chapter*{Abstract}


We provide an explicit formula for primitive elements in the Hall algebras of nilpotent representations of cyclic quivers.


\clearpage
\pagenumbering{arabic}
\setcounter{page}{1} 

\chapter{Introduction}
To a finitary category, Ringel attached an associative algebra which is commonly called the Ringel-Hall algebra or just the Hall algebra. The multiplication of the Hall algebra is related to the number of short exact sequences with fixed isomorphism classes. Later, Ringel showed that the structure constants are specializations of certain polynomials, called the Hall polynomials when the finitary category is the category of representations of a Dynkin quiver over a finite field. Further, J.A.Green proved that there is a topological coproduct for the Ringel-Hall algebra associated with any finitary category, and that, this coproduct gives the structure of a bialgebra in the case of the category of representation of a quiver \cite{Green}. Xiao discovered an antipode \cite{Xiao} which provides the Hall algebra with a twisted Hopf algebra structure. The twisting is related with the coproduct. One interesting property of the Hall algebra provided by Ringel is that for Dynkin quivers, there is an isomorphism between the Borel part of the quantum universal enveloping algebra and the extended Hall algebra of the category of representations of the corresponding quiver. Further, it is shown that for a general quiver without loops, the positive part of the quantum universal enveloping algebra is isomorphic to the composition algebra of the Hall algebra of the corresponding quiver. Moreover, this isomorphism is a twisted Hopf algebra isomorphism. By the Milnor-Moore theorem, any Hopf algebra generated by primitive elements is the universal enveloping algebra of the Lie algebra structure on these elements. Thus, it is an interesting problem to describe primitive elements of the Hall algebra.\par 
The Hall algebra of a cyclic quiver has been intensively studied since its incarnation. When the cyclic quiver is the Jordan quiver, the study of this Hall algebra goes back to Steinitz, Hall and Macdonald. Their work is the inspiration of Ringel's work. The Hall algebra of the Jordan quiver is a genuine Hopf algebra instead of a twisted one. Steinitz, Hall and Macdonald\cite{Macdonald} showed that there is an isomorphism between the Hall algebra of nilpotent representations of the Jordan quiver and ring of symmetric functions over the same field. Based on this result, Macdonald showed that every primitive element in the Hall algebra of the category of nilpotent representations of a Jordan quiver has the formula \begin{align*}
  \sum_{\lambda\vdash n}(1-q)...(1-q^{l(\lambda)-1})I_\lambda  
\end{align*} where $\lambda$ is a partition of $n$, $l(\lambda)$ is the length of $\lambda$, $I_n$ is the $n-dimensional$ indecomposable representation of the Jordan quiver, and $I_{\lambda} = I_{\lambda_1}\oplus...\oplus I_{\lambda_n}$. It was shown by Ringel, Hubery and Schiffmann that the Hall algebra of the category of nilpotent representations of a cyclic quiver is the tensor product of the center with its composition algebra. Hubery and Schiffmann extended Steinitz, Hall and Macdonald result, proving that there is an isomorphism between the ring of symmetric functions and the center of the Hall algebra of nilpotent representations of a cyclic quiver over a finite field \cite{Hubery}. Hubery asked a question about what primitive elements are like in the Hall algebra of nilpotent representations of a cyclic quiver\cite{Hubery}. In this paper, we answer this question, showing that primitive elements $p_n$ for case $\hat{A}_N$ have the formula as in Theorem \ref{theorem:6.3.2}
  \begin{align}
        p_n = \frac{n}{1-q^{-Nn}}\sum_{\lambda\vdash n}(-1)^{l(\lambda)+1}\frac{1}{l(\lambda)}{l(\lambda)\choose m_1(\lambda),...,m_n(\lambda)}z_{\lambda}
  \end{align} 
In the process of proving the above result, we made use of the following ingredients. The first one is the isomorphism between the Hall algebra of the Jordan quiver and the ring of symmetric functions. The second one is the recursive formula for primitive elements of the ring of symmetric functions in a base denoted as $c_\lambda(x,t)$. We are somewhat surprised to see that the formula in Prop \ref{prop:3.1.3}
\begin{align*}
    p_n = \frac{n}{1-t^{n}}\sum_{\lambda\vdash n}(-1)^{l(\lambda)+1}\frac{1}{l(\lambda)}{l(\lambda)\choose m_1(\lambda),...,m_n(\lambda)}c_{\lambda}(X;t)
\end{align*} 
is neglected in the literature. The third ingredient is the isomorphism between the ring of symmetric functions and center of Hall algebra of cyclic quiver. We can use this isomorphism to deduce result (3.1.7). This result can provide us some applications. For example, by comparing the two equations of primitive elements (3.1.7), (6.2.2) in the case of the Jordan quiver, we get an identity in proposition 6.1 which is related to the Hall number for a special case.\par 
In section 2 we will recall some basic information about the ring of symmetric functions. Special symmetric functions called cyclic symmetric functions $c_n(x)$ are introduced. These are mapped to the generators of the center of the Hall algebra. In section 3, we present the formula for primitive elements in the ring of symmetric functions in terms of the basis of cyclic symmetric functions in Proposition 3.1.3. In section 4, Some basic information about representations of quiver are introduced. In section 5, our paper will present the main theorem for the case of a cyclic quiver and in Propostion 6.3 give a detailed description of the primitive elements of the Hall algebra over nilpotent representations of a cyclic quiver. We will give some examples of applications of our formula in section 7.

\chapter{Ring of symmetric functions}
\section{Introduction}
Here we give a brief review of the ring of symmetric functions. The main reference of this chapter is \cite{Macdonald}.
\section{Ring of Symmetric functions}
Let $\mathbb{Q}[X_1,X_2,...,X_n]$ be the polynomial ring with $n$ independent variables $X_1,...,X_n$. The symmetric group $S_n$ acts on this ring by permutating variables $\sigma(X_i) = X_{\sigma(i)}$ where $\sigma$ is in $S_n, 1 \leq i\leq n$. The symmetric polynomial ring $\mathbb{Q}[X_1,X_2,...,X_n]^{S_n}$ is the fixed point locus in $\mathbb{Q}[X_1,X_2,...,X_n]$ under this action. The symmetric polynomials form a subring $$\Lambda_n = \mathbb{Q}[X_1,X_2,...,X_n]^{S_n}$$ Moreover $\Lambda_n$ is a graded ring, we have $$\Lambda_n = \bigoplus_{k\ge 0}\Lambda_n^{k}$$ where $\Lambda_n^k$ is the homogeneous symmetric polynomials of degree $k$. For $m\ge n$, consider the homomorphism $\mathbb{Q}[X_1,X_2,...,X_m] \to \mathbb{Q}[X_1,X_2,...,X_n]$ by sending $X_{n+1},...,X_m$ to 0, and each other $X_i$ to itself. This homomorphism is compatible with the actions of $S_n$ and $S_m$. Restricting this homomorphism to $\Lambda_n$, gives a homomorphism $\rho_{m,n}:\Lambda_m \to \Lambda_n$. Further restriction to $\Lambda_n^k$, we have a homomorphism $\rho_{m,n}^k:\Lambda_m^k \to \Lambda_n^k$. We now form an inverse limit $\Lambda^k = \varprojlim\Lambda_n^k$ relative to the homomorphisms $\rho_{m,n}^k$. Now let $\Lambda =\bigoplus_{k\ge 0}\Lambda^k$. We call $\Lambda$ the ring of symmetric functions. Let $\{p_n(X)\}$ represent power sum symmetric functions, that is $p_n(X)=\sum_iX_i^n$. Then $\Lambda = \mathbb{Q}[p_1,p_2,...]$ admits a Hopf algebra structure by setting 
$$\Delta(p_n(X)) = p_n(X) \otimes 1 + 1\otimes p_n(X), \quad \epsilon(p_n(X)) = 0, \quad S(p_n) = -p_n(X).$$ 
By definition, we see that $p_n$ are primitive elements of $\Lambda$. Because $p_n(X)$ are also generators of $\Lambda$, so all primitive elements of $\Lambda$ are $p_n(X)$.

\section{Cyclic Symmetric Functions}

  We now explain an important class of symmetric functions called cyclic symmetric functions, which are denoted by $c_n(X;t)$. They will be used later.
  Denote the extension $\Lambda\otimes_{\mathbb{Q}}\mathbb{Q}[t]$ by $\Lambda[t]$. Along this extension, there is a generalization of symmetric functions in $\Lambda$. In $\Lambda$, the complete symmetric functions $h_n(X)$ is defined as $$h_n(X) = \sum_{i_1\leq ...\leq i_n}X_{i_1}...X_{i_n}$$ 
  The generating function for $h_n(X)$ is defined as,
          $$H(T) = 1+ \sum_{n\ge 1}h_n(X)T^n = \prod_{i\ge1}(1-X_iT)^{-1}$$
  The relation between $H(T)$ and $p_n$ is shown in \cite{Macdonald}. Let $exp$ denote exponential function, we have
                $$H(T) = \exp{\sum_{n\geq1}\frac{p_n}{n}T^n}$$

  By following \cite{Hubery}, define a generalization $c_n(X;t) \in \Lambda[t]$ of $h_n(X)$ in the following way: 
  \begin{align} 
    \begin{split} 
    1+ \sum_{n\ge 1} c_n(X;t)T^n = \frac{H(T)}{H(tT)} = \exp{(\sum_{n\ge1}\frac{1-t^n}{n}p_nT^n)}
   \end{split} 
  \end{align}
 
  We take the natural logarithm $\ln$ of the equation (2.3.1) to get the relation between $c_n(X;t)$ and $p_n(X)$ in the following way:
  \begin{align} 
    \begin{split}
       \sum_{n\ge1} \frac{1-t^n}{n}p_nT^n = \ln(1+ \sum_{n\ge 1} c_n(X;t)T^n)
   \end{split} 
 \end{align} 
 
  By taking derivative with respect to $T$ for the above equation, we have
  \begin{align*} 
    \begin{split} 
       \sum_{n\ge1}(1-t^n)p_nT^{n-1} & = \frac{\sum_{n\ge 1} nc_n(X;t)T^{n-1}}{1+ \sum_{n\ge1} c_n(X;t)T^n}\\ .
    \end{split}
  \end{align*}

  Multiplying the denominator on both sides of the equation, we have
  \begin{align*} 
    \begin{split} 
        \sum_{n\ge1} nc_n(X;t)T^{n-1} &= (\sum_{n\ge1}(1-t^n)p_nT^{n-1})(1+ \sum_{n\ge1} c_n(X;t)T^n)\\ &= (\sum_{n\ge1}(1-t^n)p_nT^{n-1})(c_0(X;t)T^0+ \sum_{n\ge1} c_n(X;t)T^n)\\ &=(\sum_{a\ge1}(1-t^a)p_aT^{a-1})(\sum_{b\ge0} c_b(X;t)T^b)\\
        &= \sum_{b\ge 0}\sum_{a\ge 1}(1-t^a)p_ac_b(X;t)T^{a+b-1}\\ 
   \end{split}
 \end{align*} 
 
  Let $c_0(X;t) = 1$, $a+b = n$ and matching the coefficient correspondingly, we have
 \begin{align} \label{eq:2.3.3}
     nc_n(X;t) &= \sum_{a=1}^{n}(1-t^a)p_ac_{n-a}(X;t) 
 \end{align}

\chapter{Formula of $p_n$ in ring of symmetric functions}
\section{Introduction}
 Here we provide our formula to describe primitive elements in $\Lambda[t]$. Solving the above equation $(\ref{eq:2.3.3})$ recursively, we have 
\begin{lemma} The following formula for primitive elements in $\Lambda[t]$
\begin{equation}\label{eq:3.1.1}
  \begin{split}
        p_n & = \frac{1}{1-t^{n}}\sum_{k=1}^n(-1)^{k+1}\sum_{\begin{split}i_1+...+i_k &=n\\ i_1,...,i_k &\ge 1\end{split}}i_kc_{i_k}(X;t)c_{i_{k-1}}(X;t)...c_{i_1}(X;t)\\
   \end{split}
\end{equation}
\end{lemma} 
\begin{proof}
We will prove the lemma by induction.
 \begin{itemize}
     \item[1] When $n=1$, by equation (\ref{eq:3.1.1}) we have $p_1 = \frac{1}{1-t}c_1(X;t)$ which satisfies equation (\ref{eq:2.3.3}).
     \item[2] Suppose $p_j$ also satisfies equation (\ref{eq:2.3.3}) for $j\leq{n-1}$. 
     \item[3] Simplifying equation (\ref{eq:3.1.1}) we have,
\begin{align}
   \begin{split}
      (1-t^{n})p_n & = nc_n(X;t)-\sum_{i_1+i_2 = n}i_2c_{i_2}(X;t)c_{i_1}(X;t)+...\\ & + (-1)^{n+1}\sum_{i_1+...+i_n = n}i_nc_{i_n}(X;t)...c_{i_1}(X;t)\\
                   & = nc_n(X;t) - \sum_{i_1 = 1}^{n-1}\Big(\sum_{i_2=n-i_1}i_2c_{i_2}(X;t)+...\\ &+(-1)^n\sum_{i_n+...+i_2 =n-i_1}i_nc_{i_n}(X;t)..c_{i_2}(X;t)\Big)c_{i_1}(X;t)\\
                        & = nc_n(X;t) - \sum_{i_1 = 1}^{n-1}\sum_{k=1}^{n-i_1}(-1)^{k+1}\sum_{i_{k+1}+...+i_2 = n-i_1} i_{k+1}c_{i_{k+1}}(X;t)...c_{i_2}(X;t)c_{i_1}(X;t)\\
   \end{split}
\end{align}
Substituting $n-i_1 = a$ to the right hand of last equation as follows,  
\begin{equation}\label{eq:3.1.3}
  \begin{split}
     (1-t^{n})p_n & = nc_n(X;t) - \sum_{a=1}^{n-1}\sum_{k=1}^a(-1)^{k+1}\sum_{i_{k+1}+...+i_2 = a}i_{k+1}c_{i_{k+1}}(X;t)...c_{i_2}(X;t)c_{n-a}(X;t)\\
                        & = nc_n(X;t) - \sum_{a=1}^{n-1}\sum_{k=1}^a(-1)^{k+1}\sum_{j_k+...+j_1 = a}j_kc_{j_k}(X;t)...c_{j_1}(X;t)c_{n-a}(X;t)\\
  \end{split} 
 \end{equation}
  By assumption, $p_a$ satisfies (\ref{eq:3.1.1}), so $$p_a = \frac{1}{1-t^a}\sum_{k=1}^{a}(-1)^{k+1}\sum_{j_1+...+j_k=a}j_kc_{j_k}(X;t)...c_{j_1}(X;t)$$. Substituting $p_a$ in (\ref{eq:3.1.3}), we have,
  \begin{equation}
      \begin{split}
          (1-t^n)p_n = nc_n(X;t) - \sum_{a=1}^{n-1}(1-t^a)p_ac_{n-a}(X;t)
      \end{split}
  \end{equation}

The formula (\ref{eq:3.1.1}) satisfies equation (\ref{eq:2.3.3}) for the case of $n$. So by induction, the formula for $p_n$ satisfies equation (\ref{eq:2.3.3}) for all n. Because $\{c_n(X;t)\}$ is a basis of $\Lambda[t]$, by equation (\ref{eq:2.3.3}), $p_n$ are uniquely determined. Since the formula (\ref{eq:3.1.1}) also satisfies equation (\ref{eq:2.3.3}), so the formula (\ref{eq:3.1.1}) is a correct representation of $p_n$.
\end{itemize}
\end{proof}

We can simplify equation $(\ref{eq:3.1.1})$ further. The goal is to express the sum \begin{equation}
    \sum_{\begin{split}i_1+...+i_k &=n\\ i_1,...,i_k &\ge 1\end{split}}i_kc_{i_k}(X;t)c_{i_{k-1}}(X;t)...c_{i_1}(X;t)
\end{equation} as the tautology 
\begin{equation}
    \sum_{\begin{aligned}
        r_1+2r_2...+nr_n &=n\\r_1+...+r_n &=k  \end{aligned}}\sum_{l=1}^nlg_{r_1,...,r_n}^lc_1^{r_1}(X;t)...c_n^{r_n}(X;t)        
\end{equation}
Here $r_1,...,r_n\ge 0$. $g_{r_1,...,r_n}^l$ representing the cardinality of monomials $c_{i_k}(X;t)...c_{i_1}(X;t)$ with a turple of fixed multiplicity $m(c_1(X;t))=r_1,...,m(c_n(X;t))=r_n$. $l$ here is an index $1\le l\le n$. In other words, if $|\cdot|$ represents cardinality then,
\begin{align*}
       g_{r_1,...,r_n}^l = |\biggl\{(j_{k-1},...,j_1) \vdash n-l : |\{s\in \{1,...,k-1\}: j_s = 1\}| & = r_1,...,\\ |\{s\in \{1,...,k-1\}:j_s = l\}| & = r_l-1,..., \\ |\{s\in \{1,...,k-1\}: j_s = n\}| & = r_n\biggr\}|
\end{align*}

\begin{lemma}
    For a fixed $(r_1,...,r_n)$, we have $g_{r_1,...,r_n}^l = {r_1+...+r_n-1\choose r_1,...,r_l-1,...r_n}.$
\end{lemma}
\begin{proof}
 By definition 
 \begin{align*}
     g_{r_1,...,r_n}^l = |\biggl\{(j_{k-1},...,j_1) \vdash n-l : |\{s\in \{1,...,k-1\}: j_s = 1\}| & = r_1,...,\\ |\{s\in \{1,...,k-1\}:j_s = l\}| & = r_l-1,..., \\ |\{s\in \{1,...,k-1\}: j_s = n\}| & = r_n\biggr\}|
 \end{align*}
 Simply we can get that 
 \begin{align*}
    g_{r_1,...,r_n}^l & ={r_1+...+r_n-1\choose r_1}{r_2+...+r_n-1\choose r_2}...{r_{n-1}+r_n\choose r_{n-1}}\\ &= {r_1+...+r_n-1\choose r_1,...,r_l-1,...,r_n}
 \end{align*}
\end{proof}
Let $r_1,...,r_n\ge 0$, if the formula $(\ref{eq:3.1.1})$

\begin{align*}
  \begin{split}      
      p_n & = \frac{1}{1-t^{n}}\sum_{k=1}^n(-1)^{k+1}\sum_{\begin{aligned}\begin{split} i_1+...+i_k &=n\\ i_1,...,i_k &\ge 1\end{split}\end{aligned}}i_kc_{i_k}(X;t)c_{i_{k-1}}(X;t)...c_{i_1}(X;t)\\
  \end{split}
\end{align*}

  being simplified as the tautology
\begin{align*}
  \begin{split}
            & = \frac{1}{1-t^{n}}\sum_{k=1}^n(-1)^{k+1}\sum_{\begin{aligned}\begin{split} 1r_1+...+nr_n & = n\\ r_1+...+r_n & = k\end{split}\end{aligned}} \sum_{l=1}^tl g_{r_1,...,r_n}^l c_1(X;t)^{r_1}...c_n(X;t)^{r_n}\\
   \end{split}
\end{align*}
Substituting formula of $g_{r_1,...,r_n}^l$ into equation above, we have
\begin{align*}
   \begin{split}
        p_n & = \frac{1}{1-t^{n}}\sum_{k=1}^n(-1)^{k+1}\sum_{\begin{aligned}\begin{split} 1r_1+...+nr_n & = n\\ r_1+...+r_n & = k \end{split}\end{aligned}} \sum_{l=1}^tl{r_1+...+r_n-1\choose r_1,...,r_l-1,...,r_n} c_1(X;t)^{r_1}...c_n(X;t)^{r_n}\\ 
            & = \frac{1}{1-t^{n}}\sum_{k=1}^n(-1)^{k+1}\sum_{\begin{aligned}\begin{split} 1r_1+...+nr_n & = n\\ r_1+...+r_n & = k\end{split}\end{aligned}} \sum_{l=1}^tlr_l{r_1+...+r_n-1\choose r_1,...,r_l,...r_n}c_1(X;t)^{r_1}...c_n(X;t)^{r_n}\\ 
            & = \frac{1}{1-t^{n}}\sum_{k=1}^n(-1)^{k+1}\sum_{\begin{aligned}\begin{split} 1r_1+...+nr_n & = n\\ r_1+...+r_n & = k \end{split}\end{aligned}} \sum_{l=1}^tlr_l\frac{1}{k}{r_1+...+r_n\choose r_1,...,r_n}c_1(X;t)^{r_1}...c_n(X;t)^{r_n}\\
            & = \frac{1}{1-t^{n}}\sum_{k=1}^n(-1)^{k+1}\frac{1}{k}\sum_{\begin{aligned}\begin{split} 1r_1+...+nr_n & = n\\ r_1+...+r_n & = k \end{split}\end{aligned}} \sum_{l=1}^tlr_l{r_1+...+r_n\choose r_1,...,r_t}c_1(X;t)^{r_1}...c_n(X;t)^{r_n}\\
    \end{split}
\end{align*}
After summing over $l$, we have
\begin{align}
  \begin{split}
     p_n & = \frac{1}{1-t^{n}}\sum_{k=1}^n(-1)^{k+1}\frac{n}{k}\sum_{\begin{aligned}\begin{split} 1r_1+...+nr_n & = n\\ r_1+...+r_n & = k \end{split}\end{aligned}} {r_1+...+r_n\choose r_1,...,r_n}c_1(X;t)^{r_1}...c_n(X;t)^{r_n}\\
         & = \frac{n}{1-t^{n}}\sum_{1r_1+...+nr_n=n}(-1)^{r_1+...+r_n+1}\frac{1}{r_1+...+r_n}{r_1+...+r_n\choose r_1,...,r_n}c_1(X;t)^{r_1}...c_n(X;t)^{r_n}\\
   \end{split}
\end{align}
So we have the following proposition.
\begin{prop} \label{prop:3.1.3}
   Let $\lambda$ be a partition of $n$ in the form that $\lambda = (1^{m_1(\lambda)},...,n^{m_n(\lambda)})$, $m_i(\lambda)$ is the multiplicity of $i$ in $\lambda$. One explicit formula for primitive elements in $\Lambda[t]$ is
\begin{align}\label{eq:3.1.8}
    \begin{split}
        p_n = \frac{n}{1-t^{n}}\sum_{\lambda\vdash n}(-1)^{l(\lambda)+1}\frac{1}{l(\lambda)}{l(\lambda)\choose m_1(\lambda),...,m_n(\lambda)}c_{\lambda}(X;t)
    \end{split}
\end{align} where $l(\lambda)$ is the length of $\lambda$, $c_{\lambda}(X;t) = c_1(X;t)^{m_1(\lambda)}...c_n(X;t)^{m_n(\lambda)}$.
\end{prop}

The equation $(\ref{eq:3.1.8})$ is important in our later discussion. One interesting point is that the right hand of $(\ref{eq:3.1.8})$ is in $\Lambda[t]$. We will begin introducing the Hall algebra of nilpotent representations of a cyclic quiver over a finite field. 

\chapter{Hall algebra}
\section{Introduction}
 We will now introduce the Hall algebra of a finitary category, the main reference of this section is \cite{Schiffmann}
\section{Finitary category}
Shiffmann introduced a category called finitary category in \cite{Schiffmann}, we will review it here. A category is called a small category if all of its objects form a set and all of its morphisms form a set. A small abelian category $\mathcal{A}$ is called finitary if for any two objects $M,N \in \mathcal{A}$, they satisfy the condition, 
   \begin{align*} |Hom_{\mathcal{A}}(M,N)|<\infty,\quad |Ext_{\mathcal{A}}^1(M,N)|<\infty \end{align*}
   Here $Ext_{\mathcal{A}}^n (M,N)$ is the $n-th$ cohomology group. 
   Let $\mathcal{A}$ be a finitary category, we make additional assumptions that for arbitrary objects $M,N$ in $\mathcal{A}$,  there exists a number $c_{M,N}> 0$ such that 
   \begin{align*} 
        & Ext_{\mathcal{A}}^i(M,N) = \{0\}, \text{for }i> c_{M,N} \\
        & |Ext_{\mathcal{A}}^{i}(M,N)|<\infty \text{ for all } i.
  \end{align*}
   
   For any two objects $M,N$ of $\mathcal{A}$, we put $$\langle M,N \rangle =  \left(\prod_{i=0}^{\infty}|Ext_{\mathcal{A}}^{i}(M,N)|^{(-1)^{i}}\right)^{\frac{1}{2}}$$ 
   By the assumption, the quantity $\langle M,N\rangle$ is well defined because for there exists an $c_{M,N}$ such that any $i>c_{M,N}$, $Ext_{\mathcal{A}}^i(M,N) = \{0\}$ for any objects $M,N$. $\langle, \rangle$ is called multiplicative Euler form. It is also useful to introduce the symmetric Euler form $(M,N) = \langle M,N\rangle \cdot\langle N,M\rangle$. Now we can introduce the definition of the Hall algebra of a finitary category.
   
\section{Hall algebra of a finitary category}
   Here we will introduce Hall algebra of a finitary category. The main reference is in \cite{Schiffmann}. \\
   Let $\mathcal{A}$ be a finitary category and for an object $M$ of $\mathcal{A}$, let $[M]$ denote the isomorphism class of $M$. With fixed objects $M,N,R \in \mathcal{A}$, we denote $\mathbf{P}_{M,N}^{R}$ as the cardinality of all short exact sequenes $0 \rightarrow N \rightarrow R\rightarrow M\rightarrow 0$. Set $a_M = |Aut(M)|$. The Hall algebra $\Hall$ is defined to be 
                  $$\Hall = \bigoplus_{[M]}\mathbb C[M]$$ 
   as vector space where $[M]$ runs over the set of isoclasses of objects in $\mathcal{A}$ with multiplication given by 
                 $$[M]\cdot [N] = \langle M, N\rangle \sum_R\frac{1}{a_Ma_N}\mathbf{P}_{M,N}^{R}[R]$$ 
   The summation in above definition is well-defined because $|Ext^1(M,N)|<\infty$, so there are only finitely many $R$. It is known that $\Hall$ is an associative algebra with unit $i: \mathbb{C} \rightarrow \Hall$ given by $i(c) = c[0]$, where $0$ is the zero object of $\mathcal{A}$.

 \subsection{Comultiplication on $\Hall$}
    The main reference in this section is from \cite{Schiffmann} and \cite{Green}. Paraphrasing what Schiffmann said in \cite{Schiffmann}, we can say that the multiplication in $\Hall$ encodes all the ways of forming short exact sequences with object $M$ on top of object $N$. It would be interesting to have an operation which is breaking an object $R$ into two objects. By the assumption of $\mathcal{A}$, there are only finitely many extensions of two objects $M,N$. But for a fixed object $R$, there are possible infinitely many pairs $(M,N)$ such that short exact sequence $0 \rightarrow N \rightarrow R\rightarrow M\rightarrow 0$ exist. So we need to consider certain completion of $\Hall$ and $\Hall\otimes\Hall$.
    Let $K(\mathcal{A})$ denote the Grothendieck group of $\mathcal{A}$. For $\alpha \in K(\mathcal{A})$, $\Hall[\alpha] = \bigoplus_{\bar{M} =\alpha}\mathbb{C}[M]$. Here $\Bar{M}$ stands for the class of $M$ in Grothendieck group. Set
    $$\Hall[\alpha]\widehat{\otimes}\Hall[\beta] = \prod_{[M]\in \alpha, [N]\in \beta}\mathbb C[M]\otimes \mathbb C[N]$$ $$\Hall\widehat{\otimes}\Hall = \prod_{\alpha, \beta}\Hall[\alpha] \widehat{\otimes}\Hall[\beta]$$ 
    In other words, $\Hall \widehat{\otimes}\Hall$ is the space of all formal (infinite) linear combinations $\sum_{M,N}c_{M,N}[M]\otimes[N], c_{M,N}\in \mathbb{C}.$

   \begin{prop}\cite{Green}
        The following defines on $\Hall$ the structure of a topological coassociative coproduct: for any object $R\in \Hall$, $$ \Delta([R]) = \sum_{M,N}\langle M,N\rangle \frac{\mathbf{P}_{M,N}^{R}}{a_R} [M]\otimes[N] $$ with counit $\epsilon: \Hall \rightarrow \mathbb C$ defined by $\epsilon([M]) = \delta_{M,0}$.
   \end{prop}
   
\subsection{$\Hall$ is a bialgebra over hereditary finitary category}

     By the result of the previous section, $\Hall$ is an algebra and coalgebra. But in general $\Hall$ is not a bialgebra. An algebra $B(\cdot,i,\epsilon, \delta)$ is called a bialgebra, if it satisfies the following conditions:
     \begin{itemize}
      \item[1] B is a vector space over a field k.
      \item[2] There are k-linear map (multiplication) $\cdot: B\otimes B \mapsto B$ and unit $i: k\mapsto B$ such that $(B,\times, i)$ is an unital associative algebra.
      \item[3] There are $\Delta: B\mapsto B\otimes B$ and counit $\epsilon: B\mapsto k$ such that $(B, \Delta, \epsilon)$ is a coalgebra.
      \item[4] Compatibility conditions expressed by the following commutative diagrams:

    \[
     \xymatrix{
       B \otimes B \ar[r]^\cdot \ar[d]_{\cdot \otimes \Delta} & B \ar[r]^\Delta & B \otimes B \\
       B \otimes B \otimes B \otimes B \ar[rr]^{\mbox{id} \otimes \tau \otimes \mbox{id}} & & B \otimes B \otimes B \otimes B \ar[u]_{\cdot \otimes \cdot}
         }
    \]
    \[
     \xymatrix{
        B \otimes B \ar[rr]_\cdot \ar[rd]_{\epsilon \otimes \epsilon} & & B \ar[ld]_\epsilon \\
        & k \otimes k \cong k & \\
      }
    \]
    \[
     \xymatrix{
      & k \otimes k \cong k \ar[ld]_{i \otimes i} \ar[rd]^i & \\
      B \otimes B  & & B \ar[ll]^\Delta \\
     }
    \]
    \[
    \xymatrix{
       k\ar[dd]_{id}\ar[dr]_{i}\\ & B\ar[dl]_{\epsilon}\\k
    }
    \]
    \end{itemize}
     Green proved that $\Hall$ is a bialgebra when $\mathcal{A}$ is hereditary.\cite{Green} From now on we assume that $\mathcal{A}$ is hereditary, i.e., that is for any two objects $M,N \in \mathcal{A}$, $Ext^i(M,N) = 0$ for $ i \geqslant 2.$ We slightly twist the product on $\Hall\otimes\Hall:$ Because $\Hall$ is a graded algebra, so it contains homogeneous elements in the sense of graded algebra. If $x,y,z,w$ are homogeneous elements of $\Hall$ of respect weight $wt(x),wt(y), wt(z), wt(w)\in K(\mathcal{A})$, then define$$(x\otimes y)\cdot (z\otimes w) = \langle wt(y),wt(z)\rangle\langle wt(z), wt(y)\rangle(xz\otimes yw)$$ The following is proved by Green \cite{Green}.

   \begin{prop}
         (Green) Assume that $\mathcal{A}$ is hereditary. The map $\Delta:\Hall \rightarrow \Hall\widehat{\otimes}\Hall$ is an algebra homomorphism, i.e. for any two $x,y \in \Hall$, $\Delta(x\cdot y) = \Delta(x)\cdot\Delta(y)$.
   \end{prop}
   
Because the above proposition, we have that $\Hall$ is a bialgebra.

\chapter{Categories of representations of quivers}
\section{Introduction}
   In the previous section, we showed that $\Hall$ is a bialgebra. Furthermore Xiao in \cite{Xiao} proved that for the category of representations of a quiver, the extended hall algebra $\Tilde{{\mathbf{H}}}_\mathcal{A}$, which we define in section 5.3.1, is a Hopf algebra. So we introduce representations of quivers in this section. The main reference in this section is from \cite{Schiffmann}

\section{Representations of quivers}
  \begin{definition}
       A quiver $Q$ is an oriented graph with vertices and arrows between vertices. (Multiple arrows and loop arrows are allowed in $Q$.) Let $Q_0$ denote the set of vertices and $Q_1$ denote the set of arrows. For $a \in Q_1$, denote $h(a)$ the vertex $a$ points to, and $t(a)$ the other endpoint of a. 
   \end{definition}
   
   \begin{definition}
      For a quiver $Q$, a representation of $Q$ over a field $k$ is an assignment of a vector space $V_i$ over $k$ for each $i\in Q_0$ and a linear morphism $\rho_a: V_i \rightarrow V_j$ for each $a\in Q_1$ from $i$ to $j$. So, we denote a representation of a fixed $Q$ as $(\bigoplus_{i\in Q_0}V_i,(\rho_a)_{a \in Q_1})$.
    \end{definition}
    
    For a fixed quiver, a morphism between two representations $(V_i,\rho_a)$, $(V_i', \rho'_a)$ is $(f_i: V_i \rightarrow V'_i)$, $i\in Q_0$ such that the following diagram commute for any $i,j\in Q_0$ and any $a \in Q_1$ such that $h(a)=j, t(a)=i$. 
    
     \begin{figure}[!h]
        \centering
        \begin{tikzpicture}
             \matrix(m)[matrix of math nodes, row sep=3em,column sep=4em,minimum width=2em]{V_i & V_j \\ V_i' & V_j' \\};
          \path [-stealth] 
            (m-1-1) edge node [left] {$f_i$} (m-2-1);
                    edge node [above] {$\rho_a$} (m-1-2);
            (m-2-1) edge node [below] {$\rho_a$} (m-2-2);
            (m-1-2) edge node [right] {$f_j$}(m-2-2);      
       \end{tikzpicture}
    \end{figure}

   A representation is called a nilpotent representation if there exist an integer $n$, such that $\rho_{a_1} \rho_{a_2}... \rho_{a_n} = 0$ for any $a_i\in Q_0, i=1,...,n-1$, satisfying $h(a_i) = t(a_{i+1})$.\\
   A representation is called an indecomposable representation if it cannot be decomposed as a direct sum of subrepresentations. \\
   A representation is called simple if there is no subrepresentation. We denote simple representations as $S_i, i\in Q_0$
 \subsection{Representations of Jordan quiver}
 \begin{example} 
      A quiver is called a Jordan quiver if $|Q_0| = 1$ and $|Q_1| = 1$. Up to isomorphism, there is only one Jordan quiver. All indecomposable representations of the Jordan quiver is given by a $n \times n$ matrix,
         \begin{figure}[!h]
             \centering
             $\rho = 
           \begin{pmatrix}
              \mu & 1 & 0 & \cdots & 0 \\
              0 & \mu & 1 & \cdots & 0 \\
              0 & 0 & \ddots & \ddots \\
              0 & 0 & \dots & \mu & 1\\
              0 & 0 & \dots & 0 & \mu
           \end{pmatrix}$
         \end{figure}

     where $n = \text{dim } V$. Denote an indecomposable representation of dimension $n$, diagonal $\mu$ by $I_{[n; \mu]}$.  When we consider nilpotent representations of Jordan quiver, then $\rho^t=0$ for some $t\ge 0$, then the eigenvalue $\mu$ of $\rho$ is $\mu = 0$. If we introduce the notation that $I_{[n]}=I_{[n;0]}$, the nilpotent representation looks like $(V, \rho) = I_{[1]}^{\oplus r_1}\oplus\dots\oplus I_{[n]}^{\oplus r_n}$. So Krull-Schmidt theorem for representations asserts that every representation $(V, \rho) = I_{[1]}^{\oplus r_1}\oplus\dots\oplus I_{[n]}^{\oplus r_n}$ where $n = 1r_1+...+nr_n, r_i\geq 0$. This works for a field no matter that is algebraically closed or not. So if a partition $\lambda$ of $n$ is $\lambda = (1^{r_1},...,n^{r_n})$, and $I_{\lambda} = I_{[1]}^{\oplus r_1}\oplus\dots\oplus I_{[n]}^{\oplus r_n}$, then every nilpotent representation $(V,\rho) = I_{\lambda}$.
   \end{example}
   
\subsection{Representations of Cyclic quiver}
  \begin{example}
      The quiver is called the cyclic quiver of $n$ vertices, denoted as $\hat{A}_n$, is the unique (up to isomorphism) quiver with the diagram looks like the following
     
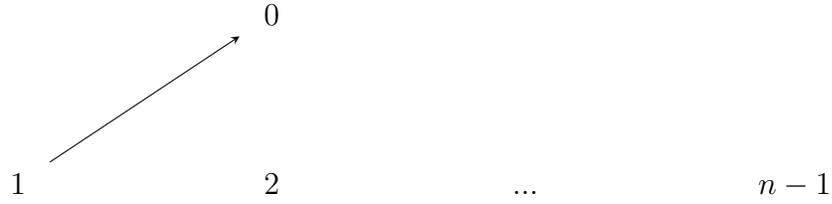
\begin{figure}[!h]
         \centering
         \begin{tikzpicture}
         \matrix(m)[matrix of math nodes, row sep=4em,column sep=6em,minimum width=2em] {& 0 & \\ 1 & 2 & ... & n-1\\};
         \path[-stealth]
           (m-2-1) edge (m-1-2);
           (m-2-2) edge (m-2-1);
           (m-2-3) edge (m-2-2);
           (m-2-4) edge (m-2-3);
           (m-1-2) edge (m-2-4);
      \end{tikzpicture}
    \end{figure}

    For $i\in \{0,1,...,n-1\}$ and $l\geq 1$. Consider $V' = \oplus_{j=i+1-l}^ike_j$ and define $x\in End(V')$ by $x(e_j) = e_{j-1}, j\neq i+1-l$ and $x(e_{i+1-l}) = 0$. Let $V_h = \oplus_{j\equiv h} ke_j, V = \oplus_{h\in \mathbb{Z}/n\mathbb{Z}}V_h$. The indecomposable nilpotent representations of $\hat{A}_n$ are all $V$ with different $i$, we denote all these representations as $I_{[i:l]}$. Then any nilpotent representation is a direct sum of $I_{[i;l]}$ with various value for $i, l$
  \end{example}
  
\section{Hall algebra of quiver representations}
   Denote category of representations of quiver $Q$ over field $k$ as $Rep_kQ$, $Rep_kQ$ is a hereditary finitary category. So by the definition, there is a Hall algebra $\mathbf H_{Rep_kQ}$ for an arbitrary quiver $Q$, we shorten the denotation $\mathbf H_{Rep_kQ}$ as $\mathbf H_{Q}$. Furthermore we denote the category of nilpotent representation of quiver $Q$ over field $k$ as $Rep_k^{nil}Q$, and the Hall algebra over $Rep_k^{nil}Q$ as $\mathbf H^{nil}_Q$.
  \subsection{Hall algebra is a Hopf algebra}
    Xiao in \cite{Xiao} proves that extended Hall algbra of category of representations of quivers are Hopf algebras. We will introduce it here.
    \begin{definition}
       Let $\mathbf{K}=\mathbb{C}[K(\mathcal{A})]$. We define an $extended$ Hall algebra $\Tilde{\mathbf{H}}_\mathcal{A} = \Hall\otimes_{\mathcal{C}}\mathbf{K}$ where by adding the relations $$\mathbf{k}_\alpha[M]\mathbf{k}_\alpha^{-1}=(\alpha, M)[M]$$ The extended Hall algebra $\Tilde{\mathbf{H}}_\mathcal{A}$ is still graded where $deg(k_\alpha)=0$. If we define the multiplication on $\Tilde{\mathbf{H}}_\mathcal{A}\otimes \Tilde{\mathbf{H}}_\mathcal{A}$ as $(x\otimes y)(z\otimes w) = xz\otimes yw$, then $\Tilde{\mathbf{H}}_\mathcal{A}$ is still a bialgebra when $\mathcal{A}$ is hereditary with comultiplication $$\Delta(\mathbf{k}_\alpha)=\mathbf{k}_\alpha\otimes \mathbf{k}_\alpha,$$ $$\Delta([R]\mathbf{k}_\alpha) = \sum_{M,N}\langle M,N\rangle\frac{1}{a_R}\mathbf{P}_{M,N}^R[M]\mathbf{k}_{[N]+\alpha}\otimes[N]\mathbf{k}_\alpha.$$ We can extend the antipode $S$ in Xiao's paper \cite{Xiao} to $S([M]\mathbf{k}_\alpha)=\mathbf{k}_{\alpha}^{-1}S([M])$, then $\Tilde{\mathbf{H}}_\mathcal{A}$ is a Hopf algebra.
    \end{definition}
    
\section{Isomorphism between Hall algebra and quantum Lie group}
   Suppose that $\mathbf{U}_v(\mathfrak{b}_+')$ is the positive half of the quantum group $\mathbf{U}_v(\mathfrak{g'})$ corresponding to $\mathfrak{g}$. The detailed information about quantum group can be found in \cite{Schiffmann} appendix A.4. Here we have the following nice property.
  \begin{prop}
  (Ringel \cite{Ringel},Green \cite{Green}) There is an embedding of Hopf algebra  $$\Phi:\mathbf{U}_v(\mathfrak{b}_+')\rightarrow \Tilde{\mathbf{H}}_Q $$ by sending $E_i\rightarrow [S_i], K_i\rightarrow \mathbf{k_{S_i}}$ for $i\in I$. The map $\Phi$ is an isomorphism when $Q$ is of finite type.
  \end{prop}
  We can restrict the above map $\Phi$ to $\mathbf{U}_v(\mathfrak{n}_+')$ which gives an embedding $\mathbf{U}_v(\mathfrak{n}_+') \rightarrow \Tilde{\mathbf{H}}_Q$. The image of this map is called composition algebra, which is the subalgebra of $\Tilde{\mathbf{H}}_Q$ generated by simple representations.

\chapter{Primitive elements in the Hall algebra}
\section{Introduction}
   From this section, we introduce an isomorphism between ring of symmetric functions and center of Hall algebra over nilpotent representation of cyclic quiver. In Jordan quiver $\hat{A}_0$, there is an formula for primitive elements in $\mathbf H_{\hat{A}_0}^{nil}$ calculated by Hall, Steinitz and Macdonald \cite{Macdonald}. By the isomorphism we will introduce, we can give an explicit formula for primitive elements for all $H_{\hat{A}_n}^{nil}$. The main reference is from \cite{Hubery}.
   
\section{Primitive element of Jordan Quiver case}
   For $\hat{A}_0$ being the Jordan quiver, $\mathbf{H}_{\hat{A}_0}$ is defined over field $k=\mathbb{F}_q$. we introduce the isomorphism between $\mathbf H_{\hat{A}_0}^{nil}$ and ring of symmetric functions. Suppose $P_{\lambda}(\boldsymbol{X},t)$ denote Hall-Littlewood functions in $\Lambda[t]$.
\begin{prop} \label{prop:5.2.1}
(Steinitz, Hall, Macdonald \cite{Macdonald}). For $\hat{A}_0$ being the Jordan quiver, the field $k=\mathbb{F}_q$. There is an isomorphism $\Phi$ between $\Lambda[t]\otimes \mathbb Q(t^{1/2})$ and $\mathbf H_{\hat{A}_0}^{nil}$ where

  \begin{align}
     \begin{split}
         \Phi : \Lambda[t] \otimes_{\mathbb{Q}} \mathbb Q(t^\frac{1}{2}) &\rightarrow \mathbf H_{\hat{A}_0}^{nil}\\ t &\mapsto q^{-1}\\ t^{n(\lambda)}P_{\lambda}(\boldsymbol{X},t) &\mapsto I_{\lambda}
     \end{split}
 \end{align} where $n(\lambda) = \sum_i(i-1)\lambda_i$.
 
\end{prop}

In $\Lambda[t]$, power sum $p_n(\boldsymbol{X})=\sum_iX_i^n$ are the primitive elements. In \cite{Macdonald}, Macdonald shows that $p_n(\boldsymbol{X}) = \sum\limits_{\lambda\vdash n}t^{n(\lambda)}\prod\limits_{i=1}^{l(\lambda)-1}(1-t^{-i})P_{\lambda}(\boldsymbol{X};t).$ Under proposition \ref{prop:5.2.1}, we have 
  \begin{align}
     \begin{split}
         p_n(\boldsymbol{X}) \mapsto \sum_{\lambda\vdash n}(1-q)...(1-q^{l(\lambda)-1})I_\lambda
     \end{split}
  \end{align} 
Because $p_n(\boldsymbol{X})$ in $\Lambda[t]$ are primitive elements, so $\sum_{\lambda\vdash n}(1-q)...(1-q^{l(\lambda)-1})I_\lambda$ are primitive elements in $\mathbf H_{\hat{A}_0}^{nil}$.

\section{Primitive elements in Hall algebra for Cyclic quiver}
Now we extend the above map to the case of cyclic quiver $\hat{A}_N$, similarly the field is $k=\mathbb{F}_q$.
\subsection{Isomorphism between $\Lambda[t]$ and central subalgebra of $\mathbf H_{\hat{A}_N}^{nil}$}
  \begin{prop}(Shiffmann\cite{Schiffmann[2]}, Hubery\cite{Hubery})

    \begin{itemize}
      \item[1] $\mathbf H_{\hat{A}_N}^{nil} \cong C_N \otimes Z_N $ as self-dual graded Hopf algebra, where $Z_N$ is the center of $\mathbf H_{\hat{A}_N}^{nil}$. $C_N$ is the composition algebra.
      \item[2] $z_N = \mathbb{Q}(q^{1/2})[z_1, z_2,....]$ where 
        \begin{align}
            z_r =(-q^{-1})^{rN}\sum_{\begin{aligned}
           &dim [M]= r\delta \\ &socle(M) \text{ is square-free}\\
       \end{aligned}}(-1)^{\text{dim End(M)}} |Aut(M)|[M]
        \end{align}
        
       Here socle (M) = $\sum_{\text{N is a simple submodule of M}} N $

       \item[3] There is a natural isomorphism 
         \begin{align}
           \begin{split}
               \Lambda[t]\otimes\mathbb{Q}(t^\frac{1}{2N}) & \rightarrow Z_N \\
                 t &\mapsto q^{-N}\\ c_r(X;t) & \mapsto z_r
           \end{split}
         \end{align}
    \end{itemize}
\end{prop}

\subsection{Primitive elements of Hall algebra of cyclic quiver}
By our equation (\ref{eq:3.1.8}), primitive elements $p_n \in \Lambda[t]$ has formula,
\begin{align}
    p_n = \frac{n}{1-t^{n}}\sum_{\lambda\vdash n}(-1)^{l(\lambda)+1}\frac{1}{l(\lambda)}{l(\lambda)\choose m_1(\lambda),...,m_n(\lambda)}c_{\lambda}(X;t)
\end{align}
Taking $p_n$ under the above isomorphism (5.3.1), we have

\begin{theorem}\label{theorem:6.3.2}
Primitive elements $p_n$ of $H_{\hat{A}_N}^{nil}$ has formula that 
\begin{align}
    p_n = \frac{n}{1-q^{-Nn}}\sum_{\lambda\vdash n}(-1)^{l(\lambda)+1}\frac{1}{l(\lambda)}{l(\lambda)\choose m_1(\lambda),...,m_n(\lambda)}z_{\lambda}
\end{align}
where $z_\lambda = z_{\lambda_1}^{m_1(\lambda)}...z_{\lambda_n}^{m_n(\lambda)}$
\end{theorem}.

\chapter{Examples}
\section{Examples of primitive elements}
   For the case of $\mathbf{H}_{\hat{A}_0}^{nil}$, we show some calculations of by using our formula. 
\begin{align*}
   \begin{split}
        z_r & = (-q)^{-r}(-1)^{dim End(I_{(r)})}|Aut({I_{(r)}})|[I_{(r)}]\\ &= (-q)^{-r}(-1)^{dim End(I_{(r)})}q^r(1-q^{-1})[I_{(r)}]\\ & = (1-q^{-1})[I_{(r)}]
   \end{split}
\end{align*} So,
\begin{align*}
    p_1 = \frac{1}{1-q^{-1}}z_1 = I_{(1)}
\end{align*}
\begin{align*}
   \begin{split}
      p_2 & = \frac{2}{1-q^{-2}}[(-1)^2{1\choose 0,1}z_2+(-1)^3\frac{1}{2}{2\choose 2,0}z_1^2]\\
             & = \frac{2}{1-q^{-2}}[z_2-\frac{1}{2}z_1^2]\\
             & = \frac{2}{1-q^{-2}}[(1-q^{-1})I_{(2)}-\frac{1}{2}(1-q^{-1})^2I_1^2]\\
             & = \frac{2}{1-q^{-2}}[(1-q^{-1})I_{(2)}-\frac{1}{2}(1-q^{-1})^2((q+1)I_1\oplus I_1+I_2)]\\
             & = I_2+\frac{q^{-1}-1}{q^{-1}}I_1\oplus I_1\\
             & = I_2+(1-q)I_1\oplus I_1\\
  \end{split}
\end{align*} 

\begin{align*}
    \begin{split}
    & p_3 = \frac{3}{1-q^{-3}}\sum_{\lambda\vdash 3}(-1)^{l(\lambda)+1}\frac{1}{l(\lambda)}{l(\lambda)\choose m_1(\lambda),...,m_n(\lambda)}z_1^{m_1(\lambda)}z_2^{m_2(\lambda)}z_3^{m_3(\lambda)}\\
        & = \frac{3}{1-q^{-3}}\Bigg\{(-1)^2\cdot 1\cdot 1z_3+(-1)^3\frac{1}{2}\cdot 2z_1z_2+(-1)^4\frac{1}{3}\cdot 3 z_1^3\Bigg\}\\
        & = \frac{3}{1-q^{-3}}\Bigg\{(1-q^{-1})I_3-(1-q^{-1})^2I_1\cdot I_2+(1-q^{-1})^3I_1^3\Bigg\}\\
        & = \frac{3}{1-q^{-3}}\Bigg\{(1-q^{-1})I_3-(1-q^{-1})^2(qI_2\oplus I_1+I_3)+(1-q^{-1})^3\big\{{2 \choose 1}_+{3 \choose 1}_+I_1^{\oplus 3}\\ & +({2 \choose 1}_++q) I_2\oplus I_1+I_3\big\}\Bigg\} \\
        & = \frac{3}{q^3-1}\Bigg\{I_3+(q-1)^2(q^2-q-1)I_2\oplus I_1+(q-1)^3(q+1)(q^2+q+1)I_1^{\oplus 3}\Bigg\}\\
    \end{split}
\end{align*}

\section{An identity of Hall number}
    In $\mathbf{H}_{\hat{A}_0}^{nil}$, the two formulas $$p_n = \sum_{\lambda\vdash n}(1-q)...(1-q^{l(\lambda)-1})I_{\lambda}$$ and $$p_n=\frac{n}{1-q^{-n}}\sum_{\lambda\vdash n}(-1)^{l(\lambda)+1}\frac{1}{l(\lambda)}{l(\lambda)\choose m_1(\lambda),...,m_n(\lambda)}(1-q^{-1})^{l(\lambda)}I_{(1)}^{m_1(\lambda)}...I_{(n)}^{m_n(\lambda)}$$ represent same elements. Suppose that 
\begin{align*}
    I_{(1)}^{m_1(\lambda)}...I_{(n)}^{m_n(\lambda)} = \sum_{\mu\vdash n}f_{(1)^{m_1(\lambda)},...,(n)^{m_n(\lambda)}}^{\mu}I_{\mu}
\end{align*} where $f_{(1)^{m_1(\lambda)},...,(n)^{m_n(\lambda)}}^{\mu}$ is the corresponding Hall number, then 

\begin{align*}
  \begin{split}
      &\sum_{\lambda\vdash n}(1-q)...(1-q^{l(\lambda)-1})I_{\lambda}\\ 
      &=\frac{n}{1-q^{-n}}\sum_{\lambda\vdash n}(-1)^{l(\lambda)+1}\frac{1}{l(\lambda)}{l(\lambda)\choose m_1(\lambda),...,m_n(\lambda)}(1-q^{-1})^{l(\lambda)}I_{(1)}^{m_1(\lambda)}...I_{(n)}^{m_n(\lambda)}\\
      &=\frac{n}{1-q^{-n}}\sum_{\lambda\vdash n}(-1)^{l(\lambda)+1}\frac{1}{l(\lambda)}{l(\lambda)\choose m_1(\lambda),...,m_n(\lambda)}(1-q^{-1})^{l(\lambda)}\sum_{\mu\vdash n}f_{(1)^{m_1(\lambda)},...,(n)^{m_n(\lambda)}}^{\mu}I_{\mu}\\
      &=\frac{n}{1-q^{-n}}\sum_{\lambda\vdash n}\sum_{\mu\vdash n}(-1)^{l(\lambda)+1}\frac{1}{l(\lambda)}{l(\lambda)\choose m_1(\lambda),...,m_n(\lambda)}(1-q^{-1})^{l(\lambda)}f_{(1)^{m_1(\lambda)},...,(n)^{m_n(\lambda)}}^{\mu}I_{\mu}
  \end{split}
\end{align*}
If we switch $\lambda,\mu$, we get that 
\begin{align*}
    \begin{split}
        &=\frac{n}{1-q^{-n}}\sum_{\mu\vdash n}\sum_{\lambda\vdash n}(-1)^{l(\mu)+1}\frac{1}{l(\mu)}{l(\mu)\choose m_1(\mu),...,m_n(\mu)}(1-q^{-1})^{l(\mu)}f_{(1)^{m_1(\mu)},...,(n)^{m_n(\mu)}}^{\lambda}I_{\lambda}\\
        &=\frac{n}{1-q^{-n}}\sum_{\lambda\vdash n}\sum_{\mu\vdash n}(-1)^{l(\mu)+1}\frac{1}{l(\mu)}{l(\mu)\choose m_1(\mu),...,m_n(\mu)}(1-q^{-1})^{l(\mu)}f_{(1)^{m_1(\mu)},...,(n)^{m_n(\mu)}}^{\lambda}I_{\lambda}
    \end{split}
\end{align*}
So we get the following theorem. 
\begin{theorem} For a fixed $\lambda$, we have the identity for $f_{(1)^{m_1(\mu)},...,(n)^{m_n(\mu)}}^{\lambda}$ that
   \begin{equation*}
        (1-q)...(1-q^{l(\lambda)-1})=\frac{n}{1-q^{-n}}\sum_{\mu\vdash n}(-1)^{l(\mu)+1}\frac{1}{l(\mu)}{l(\mu)\choose m_1(\mu),...,m_n(\mu)}(1-q^{-1})^{l(\mu)}f_{(1)^{m_1(\mu)},...,(n)^{m_n(\mu)}}^{\lambda}
\end{equation*}
\end{theorem}

\section{Primitive elements for acyclic quiver}

   For a tame acyclic quiver $Q$, tubes of homogeneous regular representation are parametrized by the set $\mathbf{Proj}k[x,y]$. So for $\rho \in \mathbf{Proj}k[x,y]$, the corresponding tubes denoted as $\tau_{\rho}$. Because prime ideals in $k[x,y]$ are $(0)$, $(f(y))$ where $f(y)$ is an irreducible polynomial in $(k[x])[y]$, and $(p,f(y))$ where $p\in k[x]$ is prime, and $f(y)$ is irreducible in $(k[x]/(p))[y]$. For $\mathbf{Proj}k[x,y]$ are homogeneous prime ideals in $k[x,y]$, not containing $k[x,y]_+ = (x,y).$ So only elements in $\mathbf{Proj}k[x,y]$ are $(0)$ and homogeneous cases for $(f(y))$. Then denote deg$\rho$ as the degree of the generator of prime ideals. $\tau_{\rho}$ is equivalent to category of nilpotent representations of Jordan quiver. Consider the Hall algebra $H_{\tau_{\rho}}$ which is isomorphic to the classical Hall algebra, so primitive elements in $H_{\tau_{\rho}}$ are, $$P_n(\rho) = \sum_{\lambda\vdash n}\left(\prod_{j=1}^{l(\lambda)-1}(1-q^{j\text{deg}{\rho}})\right)I_{\lambda}$$ For a fixed $\tau_\rho$, $P_n(\rho)=p_n$ in the Hall algebra of Jordan quiver. But the above elements are not necessarily primitive in Hall algebra $\mathbf{H}_Q$ because an object in $\tau_{\rho}$ can have preprojective subobjects and preinjective quotients. In paper by \cite{Berenstein}, the authors make conjecture 3.3 that the elements \begin{align}\begin{split}
        P_n(\rho)-\frac{1}{N(deg\rho)}\sum_{\rho'\in \mathbb{KP}^1:deg{\rho'}=deg{\rho}}P_n(\rho')\
    \end{split}\end{align} are primitive in $H_Q$, where $N(d)$ is the number of elements of $\mathbb{KP}^1$ of degree d. In paper \cite{deng2015hall}, the authors proved this conjecture to be correct. In Hall algebra of Jordan quiver, we have our formula that, $$p_n = \frac{n}{1-q^{-n}}\sum_{\lambda\vdash n}(-1)^{l(\lambda)+1}\frac{1}{l(\lambda)}{l(\lambda)\choose m_1(\lambda),...,m_n(\lambda)}z_{\lambda}$$ so we can substitute our formula into the formula (6.3.1) for $P_n(\rho)$. For a fixed $d$, $N(d)=\frac{1}{d}\sum_{k|d}\mu(\frac{d}{k})q^k$, where $\mu(r)$ is the M{\"o}bius function.
    So we have that primitive elements have formula that 
    \begin{corollary} For a tame acyclic quiver, elements $p_n$ in the following form are primitive,
    \begin{align}
        \begin{split}
            p_n = \frac{n}{1-q^{-n}}\sum_{\lambda\vdash n}(-1)^{l(\lambda)+1}\frac{1}{l(\lambda)}{l(\lambda)\choose m_1(\lambda),...,m_n(\lambda)}\big(z_{\lambda}(\rho)-\frac{1}{N(deg\rho)}\sum_{\rho'}z_{\lambda}(\rho')\big)
        \end{split}
    \end{align}
    \end{corollary}

\section{Primitive elements for Hall algebra of cyclic quiver for all case}
   
   Representation of cyclic quiver over finite field has the property that all representations are regular. So we have that 
    \begin{align}
          \mathbf H_{\hat{A}_N} & = \mathbf H_{\hat{A}_N}^{homo}\times \mathbf H_{\hat{A}_N}^{nonhomo}
    \end{align}
    Here $\mathbf H_{\hat{A}_N}^{homo}$ is isomorphic to $\mathbf {H}_{\hat{A}_0}^{nil}$, $\mathbf H_{\hat{A}_N}^{nonhomo}$ is isomorphic to $H_{\hat{A}_N}^{nil}$. So we have (7.7) as 
    \begin{align*}
        \mathbf H_Q & = \mathbf H_Q^{homo}\times \mathbf H_Q^{nonhomo}\\
             & = \mathbf H_Q^{homo}\times Z(Q)\times C(Q)
    \end{align*}
    So the center is $\mathbf H_Q^{homo}\times Z(Q)$ and the primitive elements are $(p_n, p_m,[0])$ where $p_n$ is a primitive element of $\mathbf H_{\hat{A}_0}^{nil}$ and $p_m$ is a primitive element of $\mathbf H_{\hat{A}_N}^{nil}$.
    
\section{Primitive elements for Hall algebra of 2-periodic complexes}

 Bridgeland defined a twisted Hall algebra $\mathbf{DH}(\mathcal{A})$ over $\mathbb{Z}_2-graded$ complexes for any finitary category $\mathcal{A}$ in \cite{Bridgeland}. Bridgeland proves that when $\mathcal{A}$ is finitary, hereditary, there is an isomorphism of algebra $\Phi :$ Drinfeld double $\Tilde{\mathbf{H}}_{\mathcal{A}} \rightarrow \mathbf{DH}(\mathcal{A})$. Later, Yanagida \cite{Yanagida} proves that $\mathbf{DH}(\mathcal{A})$ is actually a bialgebra and $\mathbf{DH}(\mathcal{A})$ is isomorphic to $\Tilde{\mathbf{H}}_{\mathcal{A}}$ as a bialgebra. Then we have for cyclic quiver case, we can use our method to prove primitive elements by maps $\Lambda[t]\rightarrow$ center of $\Tilde{\mathbf{H}}_{\mathcal{A}} \hookrightarrow \mathbf{DH}(\mathcal{A}).$ 
 
 \section{Examples of $z_n$}
 
    For the case of $\mathbf{H}_{\hat{A}_1}^{nil}$, since 
    \begin{align}
        z_n =(-1)^nq^{-2n}\sum_{\begin{aligned}
        &dim [M] = n\delta \\ &soc(M) \text{ is square-free}\\\end{aligned}} (-1)^{dim End(M)}a_M[M].
    \end{align}
    
   In order for $[M]$ to be square-free, $[M] = I_{[1;l_1]}\oplus I_{[0;l_0]}$ with $dim I_{[1;l_1]} = (a_1,a_2)$, $dimI_{[0;l_0]} = (b_1,b_2)$. $[M]$ only exists for \begin{itemize}
           \item[case 1] $a_1 = a_2$\\
               $[M] = I_{[1;2a]}\oplus I_{[0;2(n-a)]}$, $a_M = |Aut(I_{[1;2a]}\oplus I_{[0;2b]})| = (q-1)^2q^{n-2}$\\
           \item[case 2] $|a_1-a_2|=1$\\
                  $[M] = I_{[1;2a-1]}\oplus I_{[0;2(n-a)+1]}$, $a_M = |Aut(I_{[1;2a-1]}\oplus I_{[0;2(n-a)+1]})|=(q-1)^2q^{n-1}$
           
       \end{itemize}
    So \begin{align*}
      \begin{split}
          z_n & = q^{-2n}[\sum_{a=1}^n(-1)^n(q-1)^2q^{n-2}I_{[1;2a]}\oplus I_{[0;2(n-a)]} + \sum_{a=1}^n(-1)^{n+1}(q-1)^2q^{n-1}I_{[1;2a-1]}\oplus I_{[0;2(n-a)+1]}]\\
              & = (-q^{-1})^n(1-q^{-1})^2\sum_{a=1}^n I_{[1;2a]}\oplus I_{[0;2(n-a)]}-qI_{[1;2a-1]}\oplus I_{[0;2(n-a)+1]}
      \end{split}
    \end{align*}
\clearpage

\addcontentsline{toc}{chapter}{Bibliography}

\begin{singlespace}
	\setlength\bibitemsep{\baselineskip}
	\printbibliography
\end{singlespace}   

\end{document}